\documentclass{article}

\usepackage{amsmath,amsfonts,amssymb,amsthm,pstricks,pst-plot,pst-node,eucal}
\usepackage[dvips]{graphics}

\newtheorem{theorem}{Theorem}[section]
\newtheorem{lemma}[theorem]{Lemma}

\newtheorem{remark}[theorem]{Remark}

\newtheorem{proposition}[theorem]{Proposition}

\newcommand{\ZZ}{\mathbb Z}
\newcommand{\ZZt}{\mathbb Z_2}
\newcommand{\CC}{\mathbb C}
\newcommand{\Bn}{\mathbf n}
\newcommand{\Bz}{\mathbf z}
\newcommand{\fgl}{\mathfrak{gl}}
\newcommand{\gl}{U_q(\fgl(1|1))}
\newcommand{\Span}{\mathrm{span}}
\newcommand{\wedgedot}{\sideset{}{^\bullet}\bigwedge}
\newcommand{\Br}{\mathbb{B}}

\newcommand{\Cl}{\operatorname{Cl}}

\date{}
\begin{document}
\title{Schur-Weyl type duality for quantized
    $\fgl(1|1)$,\\ the Burau representation of braid groups\\ and
    invariants of tangled graphs}
\maketitle
\vspace{-1.5cm}
\begin{center}
{\large \sc Nicolai Reshetikhin\footnote{Supported by the NSF grant DMS-0601912,
and by DARPA}} \\
University of California at
Berkeley  and \\
KDV Institute for Mathematics,
Universiteit van Amsterdam
\bigskip

{\large \sc Catharina Stroppel}\\
     Mathematik Zentrum, Universit\"at Bonn
\bigskip

{\large \sc Ben Webster\footnote{Supported by an NSF postdoctoral
    fellowship.}}\\
Massachusetts Insitute of Technology and 
University of Oregon 
\end{center}

\begin{abstract}
  We show that the Schur-Weyl type duality between $\fgl(1|1)$ and
  $GL_n$ gives a natural representation-theoretic setting for the
  relation between reduced and non-reduced Burau representations.
\%end{quote}
\end{abstract}

\section*{Introduction}

The goal of this note is to clarify the relation between
reduced Burau representations of braid groups, non-reduced Burau representations, and the representation of the braid group
defined by $R$-matrices related to $\gl$.

A lot is known about the relation between the quantized universal
enveloping algebra $\gl)$ of the Lie superalgebra of
$\fgl(1|1)$, multivariable Alexander-Conway polynomials on links, and
the Burau-Magnus representations of braid groups.

In this paper we show that the Schur-Weyl type duality between
$\fgl(1|1)$ and $GL_n$ gives a natural representation-theoretic setting
for the relation between reduced and non-reduced Burau
representations. We use this simple fact as an excuse to sum up some
known (but partly folklore) facts about these representations and the invariants of knots.

In Section 1 we recall the definition of and basic facts about quantized $\fgl(1|1)$.  Section
2 describes the duality between $GL_n$ and $\gl$. In Section
3 we show how the Burau representation naturally reduces on the space
of multiplicities. Section 4 relates the Alexander-Conway polynomial
to the trace on the multiplicity space.

{\it Acknowledgements:} The authors are grateful for the hospitality at the Mathematics
Department of Aarhus University where this work was completed
and to the Niels Bohr grant from the Danish National
Research Foundation for the support. We are also grateful to H. Queffelec for carefully reading the
first draft of this paper and correcting an important sign error.

\section{Quantum $\fgl(1|1)$ and its representations}
\subsection{} Consider the Lie superalgebra $\mathfrak{g}=\mathfrak{gl}(1|1)$. Explicitly, this means we consider the super vector space $M$ of all complex $2\times 2$-matrices with even part $M_{0}$ spanned by the matrix units $E_{1,1}$ and $E_{2,2}$, and odd part $M_1$ spanned by the matrix units $E_{1,2}$ and $E_{2,1}$, equipped with the Lie superalgebra structure given by the super commutator. The universal enveloping super algebra $U(\mathfrak{g})$ has a quantum version $U_h(\fgl(1|1))$ defined as follows (see e.g. \cite{Kulish}):

Let $\CC[[h]]$ denote the ring of formal power series in $h$. The $\CC[[h]]$-super algebra $U_h(\fgl(1|1))$ is generated freely as a $\CC[[h]]$-algebra by (odd) elements $X, Y$ and (even) elements $G, H$ modulo the defining relations:
\[\{X,Y\}=e^{hH}-e^{-hH},\quad X^2=Y^2=0,\]
\[[G,X]=X, \ \ [G, Y]=-Y
\]
\[[H,X]=0\ \ [H,Y]=0,\ \ [H,G]=0\],
(using the common abbreviation $[A,B]:=AB-BA$ and $\{A,B\}=AB+BA$.)

$U_h(\fgl(1|1))$ is a Hopf superalgebra with the comultiplication
\[
\Delta X=X\otimes e^{\frac{hH}{2}}+e^{-\frac{hH}{2}}\otimes X,
\ \ \Delta Y=Y\otimes e^{\frac{hH}{2}}+e^{-\frac{hH}{2}}\otimes Y
\]
\[
\Delta H=H\otimes 1+ 1\otimes H, \ \ \Delta G=G\otimes 1 +1\otimes G
\]
The Hopf superalgebra is quasi-triangular with $R$-matrix
\begin{equation}\label{uni-R}
R=\exp(h(H\otimes G+G\otimes H))(1- e^{\frac{hH}{2}}X\otimes e^{-\frac{hH}{2}}Y)
\end{equation}
That is this element satisfies the following identities:
\[
\Delta(a)^{op}=R\Delta(a)R^{-1}
\]
and
\[
(\Delta\otimes \operatorname{id})(R)=R_{13}R_{23}, \ \ (\operatorname{id} \otimes \Delta)(R)=
R_{13}R_{12}
\]

There is an integral form $\gl\subset U_h(\fgl(1|1))$, which is
generated by $X, Y, G$ and the invertible element $t=e^{\frac{hH}{2}}$
as a  $\CC[e^h,e^{-h}]$-algebra.  As usual, we write $q=e^h$.

\subsection{}
Recall that there is up to isomorphism precisely one irreducible $\mathfrak{gl}(2)$-module of a fixed dimension $n$ (for instance the natural representation for $n=2$). In contrast, the algebra $U_q(gl(1|1)$ has $2$-(complex) parameter family
of irreducible representations on $\CC^{1|1}$  for $z\in\mathbb{C}^*,n\in\mathbb{C}$ denote by $V_{z,n}$ the
irreducible $2$-dimensional representation $V_{z,n}=\CC v\oplus \CC u$
with $v$ even and $u$ odd such that
\begin{equation}\label{2-irrep}
Xv=u, \ \ Yv=0, \ \ Gv=nv,  \ \ tv=zv,
\end{equation}
(from which $Xu=0$, $Gu=(n+1)u$, $Yu=(z^2-z^{-2})v$ and $tu=zu$ follows).
Obviously, one can also consider the representation $\Pi V_{z,n}$ with the parity of the elements reversed. The representation $\Pi V_{z,n}$ can be realized as $\epsilon\otimes V_{z,n}$ where
$\epsilon$ is an odd one-dimensional representation. These representations and their tensor products will in fact be essentially the only $\fgl(1|1)$-representations
of interest to us. For more details on the representation theory see e.g. \cite[\S 11]{V}.

\subsection{}

Let $V$ be a finite dimensional representation of $\gl$.
It decomposes into a direct sum of weight spaces for $G$,
\[
V=\bigoplus_{n\in \CC} V(n).
\]
Note that we do not assume the weights to be integral. As usual, the elements $X$ and $Y$ act from one weight space to another
\[
X:V(n)\to V(n+1), \hspace{.5in} Y: V(n)\to V(n-1),
\]
and we have $X^2=Y^2=0$. Hence, $V$ can be viewed as a complex with two
differentials acting in opposite directions.  The DeRham complex of
any K\"ahler manifold carries an action of $\fgl(1|1)$, such that the element $H$ acts
as the  Laplace operator.
Thus, the algebra $U_h(\fgl(1|1))$, and for the same reasons
$U(\fgl(1|1))$, is in a certain sense, an abstraction of the
structures of Hodge theory.

These are, in fact, isomorphic as algebras; the difference between
them lies in the action of the differential on $V\otimes
W$: the usual, diagonal action for $U(\fgl(1|1))$, the comultiplication for $\gl$ gives another action.

Alternatively, any $\fgl(1|1)$-representation can be thought of as a matrix
factorization 
with extra structure (primarily, an upgrade of
the $\ZZt$-grading to a $\ZZ$-grading).  The underlying super vector space
remains unchanged, with $X+Y$ giving the differential,
and $$(X+Y)^2=\{X,Y\}= t^2-t^{-2}$$ as the potential.

\section{The decomposition of the tensor product}
\subsection{}
Let $\Cl_N$ be the Clifford algebra (over $\mathbb{C}$) with $2N$ generators:
\[
\Cl_N=\langle a_i, b_i, i=1,\dots N| \{a_i,a_j\}=\{b_i,b_j\}=0,\{a_i,b_j\}=\delta_{ij}\rangle
\]
The algebra $\Cl_N$ has an irreducible $2^N$-dimensional representation
$U_N$ generated by a cyclic vector $v$ with $b_iv=0$.
We might identify the basis vectors with the set of $\{0,1\}$-sequences of length $N$, such that $v=(0,0,0,\ldots 0)$ and $a_i$ annihilates all basis vectors $S=(s_1,s_2,\ldots s_N)$ with $s_i=1$, and otherwise sends $S$ to $(-1)^{\sum_{k=1}^{i-1}s_k}S'$ where $S'$ differs from $S$ exactly in the
$i^{th}$-entry. If we consider the subspace $U=\Span{\langle a_1,\ldots,a_n\rangle}$ of $\Cl_N$ then there is a
natural isomorphism of graded vector spaces:
\begin{eqnarray}\label{UN}
U_N&\longrightarrow&\bigwedge^\bullet U\\
S=(s_1,s_2,\ldots s_N)&\longmapsto&a_{j_1}\wedge a_{j_2}\wedge \cdots \wedge a_{j_k}
\end{eqnarray}
where $s_{j_1}, s_{j_2},\ldots, s_{j_k}$ are precisely the $1$'s appearing (in this order) in $S$. The action of $a_i$ gets turned into $x\mapsto a_i\wedge_x$).

In case $N=1$, $U_N$ is $2$-dimensional, and the irreducible $2$-dimensional representation \eqref{2-irrep} is obtained by pulling  back the Clifford algebra action to $\gl$ via the algebra homomorphism $\gl\to \Cl_1$
\[
X\mapsto a_1, \ \ Y\mapsto (z-z^{-1})b_1, \ \ t\mapsto z, \ \ G\mapsto n+ a_1b_1
\]
This formalism can be extended to the $N$-fold tensor product (via the comultiplication $\Delta$) of these representations:
\begin{proposition}
Let $V(\Bn, \Bz)=V_{z_1,n_1}\otimes \cdots \otimes V_{z_N,n_N}$. Then the mapping
\begin{equation*}
\begin{array}{llllll}
X&\mapsto& \sum_{i=1}^N z_1^{-1}\dots z_{i-1}^{-1}z_{i+1}\dots z_Na_i,&t&\mapsto&z_1\dots z_N,\\
Y&\mapsto&\sum_{i=1}^N  z_1^{-1}\dots z_{i-1}^{-1}(z_i^2-z_i^{-2})z_{i+1}\dots z_Nb_i&G&\mapsto&\sum_{i=1}^N (n_i+ a_ib_i).
\end{array}
\end{equation*}
defines uniquely an algebra homomorphism $\Phi_{\Bn,\Bz}:\gl\to \Cl_N$. Pulling back via this map the representation $U_N$ of $\Cl_N$ gives the tensor product representation $V(\Bz, \Bn)$.
\end{proposition}

\begin{proof}
One easily verifies that the map is compatible with the relations of $\gl$. The second statement follows then also by explicit calculations.
\end{proof}

\subsection{}
The vector $v_N=v\otimes \dots\otimes v\in V(\Bz, \Bn)$ is a lowest weight vector of lowest weight $\lambda=\sum_{i=1}^N n_i$, i.e. $Yv_N=0$ and $Gv_N=\lambda v_N$.

The subspaces $U=\Span\langle a_1,\ldots,a_N\rangle$ and
$U'=\Span\langle b_1,\ldots, b_N\rangle$ of $\Cl_N$ can be paired via $U\otimes U'\to \CC$, $a_j\otimes b_i\mapsto \delta_{i,j}$.  Abbreviate $\Phi=\Phi_{\Bn,\Bz}$ and let $W=(\CC \Phi(Y))^\perp$ and $W'=(\CC \Phi(X))^\perp$.

\begin{lemma}
Let $z:=z_1z_2\cdots z_N$. Assume $z^2-z^{-2}\not=0$. Then
\begin{enumerate}
\item $U=\CC \Phi(X)\oplus W$, and $U'=\CC \Phi(Y)\oplus W'$.
\item The subspaces $W,W'$ generate\footnote{Elements of these
spaces are elements of the associative algebra $Cl_N$.} a subalgebra $C(X,Y)$ of $\Cl_N$ isomorphic to
  $\Cl_{\dim(W)}$, which is the super-commutant of the subalgebra generated by
  $X$ and $Y$.
\end{enumerate}
\end{lemma}

\begin{proof}
  The inclusion $U\subseteq\CC \Phi(X)\oplus W$ holds by
  definition. For the inverse it is enough to find (for $1\leq i\leq
  N$) $\beta_i\in\CC$ such that $a_i-\beta_i\Phi(X)\in W$. One easily
  verifies that $\beta_i=\frac{z^2(1-z_i^{-4})}{z^2-z^{-2}}$ does the
  job. The sum is direct, since an element $u$ in the intersection is
  of the form $u=\alpha\sum_{i=1}^N\gamma_i^2a_i$ where $0=
  \alpha\sum_{i=1}^N\gamma_i^2(z^2-z^{-2})=0$, hence with our
  assumption $\alpha=0$ and so $u=0$. The argument for $U'$ is similar. Part 1 follows.

  Now $C(X,Y)$ is clearly
  contained in the commutant of $X$ and $Y$.
  Since $\dim \langle X, Y\rangle=4$, and the action of this
  subalgebra on $U_N$ is semi-simple, by $2^{N-1}$ copies of the
  unique $2$-dimensional irreducible representation of $\langle X, Y\rangle$.  Thus,
  its commutant is of dimension $2^{2(N-1)}$.  Since $C(X,Y)$ has this
  dimension, it must be the entire commutant,  obviously isomorphic to the Clifford algebra as claimed.
\end{proof}

In order to find the super-commutant not just of $\Phi(X)$ and $\Phi(Y)$, but all of $\gl$, we must find the subalgebra which also commutes with $\Phi(G)$.

\begin{proposition}[Schur-Weyl duality]
  Let still $z^2-z^{-2}\not=0$. The subalgebra of $\Cl_W$ commuting with $\Phi(G)$ is that of Euler degree
  0, i.e. that generated by elements of the form $W\cdot W'$.  There
  is a natural map $U(\fgl(W))\to \Cl_W\subset \Cl_N$ whose image is
  this subalgebra.
\end{proposition}
\begin{proof}
  The first statement is obvious. Recall that $W$ and $W'$ generate a Clifford algebra, say with generators $a_i'$, $b_i'$. Note that $W\cdot W'$ forms a Lie subalgebra of $\Cl_W$ isomorphic to
  $\fgl(W)$ (by mapping $a_i'b_j'$ to the matrix unit $E_{i,j}$),
  hence this extends to an algebra map $U(\fgl(W))\to \Cl_W\subset \Cl_N$
  The image of this map is precisely the commutant, because by the PBW theorem for Clifford algebras, the subspace of Euler
  degree 0 is that of the form $\bigoplus_nW^n\cdot
  (W')^n=\bigoplus_n(W\cdot W')^n$.
\end{proof}

Under the action of $\Cl_W$, $U_N$ decomposes into two copies of
$U_W=\wedgedot W$, one with parity reversed. Thus, $V(\Bz,\Bn)$ is
completely decomposable and, up to grading shift and parity-reversal,
the summands are precisely the 2-dimensional simple modules from above.  Of course, the highest weight
vector $v_N$ generates a copy of $V_{z,\lambda}$, so all simple submodules
must be of the form $V_{z,\lambda+k}$ for some $k$ (possibly with parity
reversed). Thus,
\begin{proposition}[Tensor space decomposition]\label{intertw}\hfill\\
  The multiplicity space of $V_{z,\lambda+k}$ in $V(\Bn,\Bz)$ is the space of
  weight $k$ (for $G$) in $U_W$.  That is
\begin{equation}\label{decomp}
V(\Bz,\Bn)\simeq  \bigoplus_{k=0}^{N-1} \sideset{}{^k}\bigwedge W \otimes \Pi^kV_{z, \lambda+k}
\end{equation}
where $\Pi$ is the shift of parity, $\Pi^2=id$.
\end{proposition}

\subsection{} This decomposition of the tensor product can be made
more explicit if we chose a basis $c_i, i=1,\ldots, N-1$  in the subspace $W\subset U$ complementary to $\CC\Phi(X)$, hence fixing a decomposition $U=\CC\Phi(X)\oplus_{i=1}^{N-1}\CC c_i$. From now on we will just write $X$, $Y$ instead of $\Phi(X)$, $\Phi(Y)$.

\begin{lemma}
We have the following formulas
\begin{eqnarray*}
Xc_{i_1}\dots c_{i_k}w&=&(-1)^kc_{i_1}\dots c_{i_k}X w\\
Yc_{i_1}\dots c_{i_k}w&=&(-1)^kc_{i_1}\dots c_{i_k}Y w +\sum_{j=1}^k y_{i_j}(-1)^{j-1}c_{i_1}\dots \widehat{c_{i_j}}\dots c_{i_k}w
\end{eqnarray*}
where $w\in U$ and the $y_i$'s are defined by $Yc_i+c_i Y=y_i$.
\end{lemma}
\begin{proof}
Obvious.
\end{proof}

For a vector $v\in U$ define
\[
(v)_{i_1,\dots, i_k}=c_{i_1}\dots c_{i_k}v+
(-1)^k\frac{1}{z-z^{-1}}\sum_{a=1}^k y_{i_a}(-1)^{a-1} c_{i_1}\dots \widehat{c_{i_a}}\dots c_{i_k}v
\]

\begin{proposition}The space
\[
V_{i_1,\dots, i_k}=\CC (v_N)_{i_1,\dots, i_k}\oplus \CC X(v_N)_{i_1,\dots, i_k}
\]
where $v_N$ is the highest weight vector (see section 2.2), is an irreducible submodule isomorphic to $V_{z, (\sum_{i=1}^Nn_i)+k}$. This submodule corresponds to the monomial
$c_{i_1}\wedge\dots\wedge c_{i_k}$ in the decomposition
(\ref{decomp}).
\end{proposition}
\begin{proof} We have
\[
X(v_N)_{i_1,\dots, i_k}=(Xv_N)_{i_1,\dots, i_k}
\]
and since $Yv_N=0$, we have
\[
Y(v_N)_{i_1,\dots, i_k}=0, \ \ YX(v_N)_{i_1,\dots, i_k}=(z^2-z{-2})(v_N)_{i_1,\dots, i_k}
\]
The statement follows directly from the action of $t$ and $G$ and \eqref{2-irrep}.
\end{proof}

\section{The relation to the Burau representation}
\subsection{}The action of the universal $R$-matrix \eqref{uni-R} in the tensor product representation
$V_{z_1,n_1}\otimes V_{z_2,n_2}$ can easily be computed explicitly.
Namely, in terms of the weight basis (by abuse of language we use the basis $\{v,Xv\}$ for either module) this {\it right} action looks as follows:
\begin{align*}
  R(v\otimes v)&=z_1^{2n_2}z_2^{2n_1}v\otimes v \\
  R(v\otimes Xv)&=z_1^{2n_2+2}z_2^{2n_1}v\otimes
  Xv-z_2^{-1}z_1(z_2^2-z_2^{-2})z_1^{2n_2+2}z_2^{2n_1}Xv\otimes v\\
  R(Xv\otimes v)&=z_1^{2n_2}z_2^{2n_1+2}Xv\otimes v\\
  R(Xv\otimes Xv)&=z_1^{2n_2+2}z_2^{2n_1+2}Xv\otimes Xv
\end{align*}
In the tensor product basis $v\otimes v, v\otimes Xv, Xv\otimes v, Xv\otimes Xv$ it produces the $4\times 4$ matrix $R^{(z_1,z_2)}=z_1^{-2n_2}z_2^{-2n_1}(R)$,
\begin{equation}
R^{(z_1,z_2)}=\left[\begin{array}{cccc} 1 & 0 & 0 & 0 \\
0 & z_1^2 & -(z_2^2-z_2^{-2})z_1^3z_2^{-1}  & 0 \\
0 & 0& z_2^2 & 0 \\
0 & 0 & 0 & z_1^2z_2^2 \end{array}\right]
\end{equation}
with
\begin{equation}
{R^{(z_1,z_2)}}^{-1}=
\left[\begin{array}{cccc} 1 & 0 & 0 & 0 \\
0 &z_1^{-2} & (z_2^2-z_2^{-2})z_2^{-3}z_1  & 0 \\
0 & 0&  {z_2^{-2}} & 0 \\
0 & 0 & 0 & z_1^{-2}z_2^{-2} \end{array}\right]
\end{equation}

\subsection{}
Consider the groupoid of braids whose strands are labeled by elements in $\mathbb{C}\times\mathbb{C}^*$. Each $N$-braid with colors $(z_i,n_i)$, $1\leq i\leq N$ on its $N$ strands defines a morphism from the tuple $(\Bz,\Bn)$ to the permuted tuple $(\sigma\Bz,\sigma\Bn)$ given by the braid. Assigning to a tuple $(\Bz,\Bn)$ the representation $
V(\Bz,\Bn)=V_{z_{1},n_{1}}\otimes \dots \otimes V_{z_{N},n_{N}}
$
and to the single (positive) braid $\beta_i$ with strands colored by $a:=(z_i,n_i)$ and $b:=(z_{i+1},n_{i+1})$ the mapping
\[
\pi(\beta_i)(a,b):\quad V(\Bz,\Bn)\to V(s_i\Bz,s_i\Bn)
\]
\begin{equation}\label{pi}
\pi(\beta_i)(a,b)=-z_1^{-3}z_{i+1}^{-1}P_{i,i+1}\circ \Big(1\otimes \cdots\otimes 1\otimes R^{(z_1,z_2)}\otimes 1\otimes \cdots \otimes 1\Big)
\end{equation}
defines a representation $\pi$ of the colored braid groupoid. Here $R^{(z_1,z_2)}$  is as above, hence up to a multiple, the universal
$R$-matrix \eqref{uni-R} acting on $V_{n_{\sigma_1}, z_{\sigma_1}}\otimes V_{n_{\sigma_2}, z_{\sigma_2}}$, and $P_{i,i+1}$ is the flip map of simply swapping the two tensor factors as $x\otimes y\mapsto (-1)^{\overline{x}\overline{y}}y\otimes x$. To verify the claim note that the braid relations amount to the relations
\small
\begin{eqnarray*}
\pi(\beta_i)(a,b)\circ\pi(\beta_{i+1})(a,c)\circ\pi(\beta_i)(b,c)&=&\pi(\beta_{i+1})(b,c)\circ\pi(\beta_i)(a,c)\circ\pi(\beta_{i+1})(a,b),
\\
\pi(\beta_i)(a,b)\circ\pi(\beta_j)(c,d)&=&\pi(\beta_j)(c,d)\circ\pi(\beta_i)(a,b).
\end{eqnarray*}
\normalsize
for $j\not={i-1,i,i+1}$ and $a,b,c,d$ arbitrary colors. These relations can easily be checked by direct calculations.
In particular, the subgroup $\Br_\Bz$ of the braid group that preserves $(\Bz,\Bn)$ acts on $V(\Bz,\Bn)$. Because the operators $\pi(\beta_i)$ commute with the action of $\gl$, the action is determined by the action on multiplicity spaces.

The first interesting multiplicity space is $W$ considered as a subspace of $U$:

\begin{proposition}
\label{Burau}
 In the case where $z_1=\dots =z_n$, this braid group representation on $U$ is
 isomorphic to the Burau representation.
 Similarly, the action on $W$ gives rise to the reduced Burau representation in case $z_1=\dots
 =z_n$.
\end{proposition}

\begin{proof} Choose the basis $b_i=v\otimes v\otimes\cdots v\otimes Xv\otimes v\otimes\cdots\otimes v$, where $X$ is applied to the $i$-th factor. Then $\pi(\beta_i)$ defined in (\ref{pi}) acts on this basis as follows: $\pi(\beta_i)b_j=b_j$ for $j\not=i,i+1$, and on $b_i$ and $b_{i+1}$ it acts as the matrix
\begin{equation}
-z_i^{-3}z_{i+1}^{-1}\begin{pmatrix}
0&z_{i+1}^2\\
z_i^2 &  -(z_{i+1}^2-z_{i+1}^{-2})z_i^3z_{i+1}
\end{pmatrix}
\end{equation}
Change the basis as $b_i=A_ib'_i$ where $A_{i+1}/A_i=-z_{i+1}z_i$, then  $\pi(\beta_i)(b_j')=b_j'$ for $j\not=i,i+1$ and $\pi(\beta_i)$ acts on $b_i',b_{i+1}'$ by
\begin{equation}\label{mmag}
\begin{pmatrix}
0&z_i^{-4}\\
1 &  (1-z_{i+1}^{-4})
\end{pmatrix}
\end{equation}
In case $z_i=z_j$ for any $i,j$, we set we $t:=z_i^{-4}$ and obtain that $\pi(\beta_i)$ acts on $b_j$ an an idenity when $j\neq i, i+1$ and on $b_i$ and $b_{i+1}$ by the matrix:
\begin{equation*}
\begin{pmatrix}
0&1\\
t^{-1} &1-t^{-1}
\end{pmatrix}
\end{equation*}
But this is exactly the  Burau representation, see for example \cite[p.118 Example 3]{Birman}. The invariant subspace is $\CC Xv$. The reduced Burau representation acts in the quotient space $W=U/\CC Xv$.
\end{proof}


In general, we obtain a colored version 
of the Magnus representation of $\Br_\Bz$
obtained from an action on the free group on $N$ generators (see \cite[p.102 ff]{B} for the non-colored version and for colored version see \cite[Section 4]{Con}) and thus, Gassner representation of the pure braid group. In other words we proved the following.

\begin{theorem}
\label{rem:Magnus}
The mapping $\beta_i\mapsto \pi(\beta_i)(a_i,a_{i+1})$ gives the Magnus-Gassner representation of the pure 
braid group.
\end{theorem}

\section{Multivariable Alexander-Conway polynomial}

In this section we will use Theorem \ref{intertw}
to obtain the Alexander-Conway polynomial of a knot in terms
of $R$-matrices for quantum $\fgl(1|1)$. These results are very closely related to the results in \cite{KS} and \cite{M}.

\subsection{} To construct invariants of links and tangled graphs let
us start with the explicit decomposition of the two-folded tensor
product. We abbreviate $$\gamma:=\left({z_1^2z_2^2-z_1^{-2}z_2^{-2}}\right)^{-1},$$ (assuming from now on this inverse exists). The following linear maps explicitly describe the decomposition of the
tensor product of two generic irreducible two-dimensional representations:
\[
\varphi:\quad V_{z_1z_2, n+m}\oplus V_{z_1z_2, n+m+1} \to V_{z_1,n}\otimes V_{z_2,m}
\]
and
\[
\psi:\quad V_{z_1,n}\otimes V_{z_2,m}\to V_{z_1z_2, n+m}\oplus V_{z_1z_2, n+m+1}
\]
We denote by $w_1, Xw_1$ (resp. $w_2, Xw_2$) the standard basis in $V_{z_1z_2, n+m+1}$ and in  $V_{z_1z_2, n+m}$, and by $v_1, Xv_1$ (resp. $v_2, Xv_2$) the standard basis in $V_{z_1, n}$ and in  $V(z_2, m)$, respectively. Then the maps are defined as follows:

\begin{eqnarray*}
\varphi (w_1)&=&v_1\otimes v_2, \\
\varphi (Xw_1)&=& z_2 Xv_1\otimes v_2+(-1)^nz_1^{-1}v_1\otimes Xv_2, \\
\varphi (w_2)&=&(-1)^{n+1}{z_1^{-1}}(z_2^2-z_2^{-2})\gamma
Xv_1\otimes v_2+ z_2(z_1^2-z_1^{-2})\gamma v_1\otimes Xv_2, \\
\varphi (Xw_2)&=&Xv_1\otimes Xv_2
\end{eqnarray*}
and
\begin{eqnarray*}
\psi(v_1\otimes v_2)&=&w_1,\\
\psi(Xv_1\otimes v_2)&=&z_2(z_1^2-z_1^{-2})\gamma Xw_1+(-1)^{n+1}z_1^{-1} w_2,\\
\psi(v_1\otimes Xv_2)&=&(-1)^nz_1^{-1}(z_2^2-z_2^{-2})\gamma Xw_1+z_2w_2,\\
\psi(Xv_1\otimes Xv_2)&=&Xw_2
\end{eqnarray*}

One easily verifies that they are inverse to each other:
\[
\psi\circ \varphi=\operatorname{id}_{V\otimes V}, \ \ \varphi\circ \psi =\operatorname{id}_{V\otimes V}
\]

Let $P_0, P_1\in \operatorname{End}(V_{z_1z_2, n+m}\oplus V_{z_1z_2, n+m+1})$
be the natural orthogonal projections to the first and the second summand respectively.

For any $A\in \operatorname{End}(M)$ with $M$ an arbitrary super space, define the super trace to be
$\operatorname{str}(A)$ to be the trace of $A$ restricted to the even part of $M$ minus the trace of $A$ restricted to the odd part of $M$. For instance, if $M=V$, then $\operatorname{str}(A)=A_{v,v}-A_{Xv,Xv}$ where $v, Xv=u$ is the weight basis in $V$.

We have the following identities for the super traces:
\begin{eqnarray}
\operatorname{str}_2(\phi P_0 \psi)&=& z_2^2(z_1^2-z_1^{-2})\gamma\operatorname{id}_{V(z_1,n)}, \nonumber\\
\operatorname{str}_2(\phi P_1 \psi)&=& -z_2^2 (z_1^2-z_1^{-2})\gamma\operatorname{id}_{V(z_1,n)}, \nonumber\\\
\operatorname{str}_1(\phi P_0 \psi)&=& z_1^{-2}(z_2^2-z_2^{-2})\gamma\operatorname{id}_{V(z_2,n)}, \nonumber\\\
\operatorname{str}_1(\phi P_1 \psi)&=& -z_1^{-2} (z_2^2-z_2^{-2})\gamma\operatorname{id}_{V(z_2,n)}\label{ambi-proj},
\end{eqnarray}
Here $\operatorname{str}_{1,2}$ are partial super traces:
\[
\operatorname{str}_2(a\otimes b)=a\operatorname{str}(b), \ \
\operatorname{str}_1(a\otimes b)=\operatorname{str}(a)b
\]
The matrix $PR^{(z,z)}$ has the spectral decomposition:
\[
PR^{(z,z)}=z^2\phi P_0 \psi-z^{-2} \phi P_1 \psi.
\]

We also have
\begin{equation}\label{crossing}
\operatorname{str}_2(PR^{(z,z)})=z^2 \operatorname{id}, \ \ \operatorname{str}_2((PR^{(z,z)})^{-1})=z^2 \operatorname{id},
\end{equation}
and it is easy to check that these identities agree
with the spectral decomposition and super trace identities above.

\subsection{}

Let $\pi$ be the representation from above and $\beta$ a braid. The partial trace $\operatorname{tr}_{23\dots N}(\pi(\beta))$ is the evaluation of a central element in $U_h(\fgl(1|1))$ in the irreducible representation $V_{z_1,n_1}$. This is a general fact
about the construction of link invariants from quasitriangular
Hopf algebras \cite{RT}, \cite[Definition 2.1]{GeerP}. Therefore this partial trace is proportional to
the identity. We will write
\[
\operatorname{str}_{23\dots N}(\beta)=\langle\operatorname{str}_{23\dots N}(\pi(\beta))\rangle I_1
\]
where $I_1$ is the identity operator in $V(z_1,n_1)$.

\begin{theorem}\label{trace}
Abbreviating $z=z_1\dots z_N$, the following holds:
  \begin{eqnarray*}
    \langle \operatorname{str}_{23\dots N}(\pi(\beta))\rangle
    &=&{z_1^2-z_1^{-2}\over z^2-z^{-2}}z^2\sum_{m=0}^{N-1}
    \operatorname{tr}_{\wedge^m W}(\pi_W(\beta))
  \end{eqnarray*}
\end{theorem}

\begin{proof}
This theorem follows immediately from Proposition \ref{intertw}. The decomposition of the  tensor product $V(\Bz,\Bn)$ defines linear maps $f_m: \wedge^m W\to \wedge^m W$ for each element $f\in \operatorname{End}(V(\Bz,\Bn))$ . Using the explicit formulae for the decomposition
of two irreducible $2$-dimensional representations from
the previous section and the formulae for partial traces $\operatorname{str}_a(\phi P_b \psi)$
from the previous subsection we arrive to the identity:
\[
\langle \operatorname{str}_{23\dots N}(f)\rangle=
\sum_{m=0}^{N-1} \operatorname{tr}_{\wedge^nW}(f_m){z_1^2-z_1^{-2}\over z^2-z^{-2}}z^2
\]

\end{proof}

Let $\widehat{\beta}$ be the link which is the closure
of the braid $\beta$. We number its connected components by $1\leq i\leq k$ and denote by $w_i(\beta)$ the winding number of
the $i$-th component of $\widehat{\beta}$. Then the following holds:

\begin{theorem}The function
\begin{equation}\label{inv-k}
\tau(\beta)=\langle \operatorname{tr}_{23\dots N}(\beta)\rangle z^{2\sum_{i=1}^k w_i(\beta)}
\end{equation}
is an invariant of the link $\widehat{\beta}$.
\end{theorem}

\begin{proof}
  We have to verify the invariance with respect to Markov moves.  The
  invariance with respect to the first Markov move means $\tau(\sigma
  \beta \sigma^{-1})=\tau(\beta)$. But this identity follows
  immediately from the conjugation invariance of the ordinary trace and Theorem
  \ref{trace}.
  The second Markov move means that $\tau(\beta
  s_{n-1}^{\pm 1})=\tau(\beta)$ where $\beta$ is a braid which has no
  factors $s_{n-1}^{\pm 1}$. But this identity follows immediately
  from the property (\ref{crossing}) of $R$-matrices.
\end{proof}

As it was shown in \cite{M} the invariant (\ref{inv-k}) is the Alexander-Conway polynomial $\Delta_{z_1,\dots, z_N}$:
\[
\langle \operatorname{tr}_{23\dots N}(\pi(\beta))\rangle ={z_1^2-z_1^{-2}\over z^2-z^{-2}}z^2
\Delta_{z_1,\dots, z_N}(\widehat{\beta})
\]

\begin{remark}{\rm
For any $\gl$-linear map $f\in \operatorname{End}(V(z_1,n_1)\otimes V(z_2,n_2))$ we have the following identity.
\[
z_1^{-2}\operatorname{str}_2(f)=z_2^2\operatorname{str}_1(f)
\]
This follows immediately from \eqref{ambi-proj}.
This property is a projective version of the ambidextrous \cite{ambi}  property of $\gl$-modules. Using this
formula  the Alexander-Conway polynomial in
terms of $\gl$ can be written as a state sum and
as a state sum, it can be extended to invariants of framed graphs (see \cite[Section 3]{KS})}
\end{remark}

\end{document}